\documentclass{article}

\usepackage{amssymb}
\usepackage{amsmath}
\usepackage{mathrsfs}
\usepackage{amsthm}
\usepackage{braket}
\usepackage{graphicx}
\usepackage{url}
\usepackage{xypic}

\newtheorem{thm}{Theorem}[section]

\newtheorem{lem}[thm]{Lemma}

\theoremstyle{definition}
\newtheorem{defn}[thm]{Definition}
\newtheorem{rmk}[thm]{Remark}

\newtheorem{condL}{Condition (L)}

\newcommand{\sA}{\mathscr{A}}

\newcommand{\sH}{\mathscr{H}}

\newcommand{\cN}{\mathcal{N}}
\newcommand{\cO}{\mathcal{O}}

\newcommand{\C}{\mathbb{C}}

\newcommand{\K}{\mathbb{K}}
\newcommand{\Q}{\mathbb{Q}}

\newcommand{\fkN}{\mathfrak{N}}
\newcommand{\fkp}{\mathfrak{p}}

\newcommand{\Fb}{{\overline{F}}}
\newcommand{\Kb}{{\overline{K}}}
\newcommand{\Fom}{E}
\newcommand{\Gamh}{\widehat{\Gamma}}

\newcommand{\GL}{\mathrm{GL}}
\newcommand{\SL}{\mathrm{SL}}
\newcommand{\GSp}{\mathrm{GSp}}

\DeclareMathOperator{\Ker}{\mathrm{Ker}}
\DeclareMathOperator{\Gal}{\mathrm{Gal}}
\DeclareMathOperator{\Aut}{\mathrm{Aut}}

\DeclareMathOperator{\nInd}{\mathrm{n-Ind}}
\DeclareMathOperator{\Hom}{\mathrm{Hom}}
\DeclareMathOperator{\Ext}{\mathrm{Ext}}
\DeclareMathOperator{\Res}{\mathrm{Res}}
\DeclareMathOperator{\rec}{\mathrm{rec}}
\DeclareMathOperator{\ad}{\mathrm{ad}}

\newcommand{\LLC}{\mathrm{LLC}}
\newcommand{\BC}{\mathrm{BC}}

\newcommand{\Tr}{\mathrm{Tr}}
\newcommand{\St}{\mathrm{St}}
\newcommand{\Sym}{\mathrm{Sym}}
\newcommand{\Std}{\mathbf{Std}}
\newcommand{\id}{\mathrm{id}}

\DeclareMathOperator{\Rep}{\mathrm{Rep}}
\DeclareMathOperator{\Mod}{\mathrm{Mod}}

\newcommand{\isomto}{\xrightarrow{\sim}}

\newcommand{\cln}{\colon}
\newcommand{\immto}{\hookrightarrow}
\DeclareMathOperator*{\tensor}{\otimes}

\newcommand{\veps}{\varepsilon}

\newcommand{\vpi}{\varpi}

\DeclareMathOperator*{\Lsum}{\boxplus}

\title{The infinite base change lifting associated to an APF extension of a $p$-adic field}
\author{Megumi Takata}

\begin{document}
	
	\maketitle
	
	\begin{abstract}
		In this paper, we construct a base change lifting for an APF extension of a mixed characteristic local field.
	\end{abstract}
	
	\section{Introduction}
	Let $p$ be a prime number.
	In this paper, we shall construct a local base change lifting
	for an almost pro-$p$ cyclic extension of infinite degree.
	The point is that the local base change lifting for a totally ramified extension
	coincides with an operation coming from the close local fields theory of Kazhdan
	under some conditions.
	
	We state the result more precisely.
	For a local field $L$ with a finite residue field, we denote by $\sA(\GL_N(L))$ the set of isomorphism classes of irreducible smooth representations of $\GL_N(L)$ over $\C$.
	We denote the Weil group of $L$ by $W_L$.
	We recall that an L-parameter of $\GL_N(L)$ is a group homomorphism $\phi \cln W_L\times\SL_2(\C) \to \GL_N(\C)$
	such that $\phi|_{W_L}$ is semi-simple and smooth and $\phi|_{\SL_2(\C)}$ is algebraic.
	Let $\Phi(\GL_N(L))$ denote the set of isomorphism classes of L-parameters of $\GL_N(L)$.	
	We note that $\Phi(\GL_1(L))$ is equal to the set $\Hom(L^\times,\C^\times)$ of smooth characters of $L^\times$.
	We denote by $\LLC_L$ the local Langlands correspondence of $\GL_N$ over $L$,
	whose existence was firstly proven by \cite{LRS1993} for $L$ of positive characteristic and by \cite{HT2001} for $L$ of characteristic zero.
	Let $F$ be a finite extension of $\Q_p$ and $\Fom$ an APF extension of infinite degree,
	in particular an almost pro-$p$ extension.
	Let $F_\infty$ be the field of norms associated with $\Fom/F$.
	We denote by $\Res_{\infty/0}$ the restriction map $\Phi(\GL_N(F)) \to \Phi(\GL_N(F_\infty))$
	with respect to the natural injection $W_{F_\infty} \immto W_F$.
	
	\begin{thm}\label{thm0}
		Suppose that the extension $\Fom/F$ is cyclic.
		Then we can construct a map $\BC_{\infty/0} \cln \sA(\GL_N(F)) \to \sA(\GL_N(F_\infty))$
		such that the following diagram is commutative:
		\[
		\xymatrix{
			\sA(\GL_N(F_\infty)) \ar[r]^{\LLC_{F_\infty}}  & \Phi(\GL_N(F_\infty)) \\
			\sA(\GL_N(F)) \ar[u]^{\BC_{\infty/0}} \ar[r]^{\LLC_F} & \Phi(\GL_N(F)) \ar[u]^{\Res_{\infty/0}}.
		}
		\]			
	\end{thm}
	
	We shall call $\BC_{\infty/0}$ \textit{the base change lifting of infinite degree}.
	We construct $\BC_{\infty/0}$ by using Arthur and Clozel's result \cite{AC1989} and close fields theory of Kazhdan \cite{Kaz1986}.
	Hence our construction is basically on the representation theory of $p$-adic groups,
	that is to say, the automorphic side.
	However, we use $\LLC$ when we prove Theorem \ref{thm2}
	by showing the corresponding statement in the terms of L-parameters,
	that is, the Galois side.
	The author expects that in the future we will be able to avoid such arguments.
		
	To construct the lifting, we shall adapt Kazhdan's theory to our setting.
	Let $L$ be a local field with a finite residue field,
	$\cO \subset L$ the ring of integers, and $\fkp \subset \cO$ the maximal ideal.
	Let $\K_l(L)$ denote the principal congruence subgroup of level $l$ of $\GL_N(L)$:
	\[
		\K_l(L) = \Ker (\GL_N(\cO) \to \GL_N(\cO/\fkp^l)).
	\]
	We denote by $\Rep(\GL_N(L))$ the category of admissible smooth representations of $\GL_N(L)$
	and by $\Rep_l(\GL_N(L))$ the full subcategory of $\Rep(\GL_N(L))$
	consisting of representations generated by their $\K_l(L)$-fixed vectors.
	
	We fix an algebraic closure $\Fb$ of $\Fom$.
	For any real number $v \geq -1$, we denote by $\Gal(\Fb/F)^v$ the $v$-th ramification group in upper numbering.
	Let $b_1 < b_2 < \cdots$ be the ramification breaks of $\Fom/F$.
	We put 
	\[
		F_n = \Fb^{\Gal(\Fb/\Fom)\Gal(\Fb/F)^{b_n}}.
	\]
	For a real number $u \geq 0$, we define
	\[
	\psi_{\Fom/F}(u) =
		\int_{0}^{u} (\Gal(\Fb/F) : \Gal(\Fb/\Fom)\Gal(\Fb/F)^v) dv.
	\]
	We take a non-decreasing sequence of non-negative integers $\{ l_n\}_{n=1}^\infty$
	satisfying the following:
	
	\begin{condL}
	$l_n \to \infty$ ($n\to \infty$) and
	$l_n \leq p^{-1}(p-1) \psi_{\Fom/F}(b_n).$
	\end{condL}
	
	Then we have a theorem that $\Rep(\GL_N(F_\infty))$ can be obtained by taking the limit of certain subcategories of $\Rep(\GL_N(F_n))$:
	
	\begin{thm}\label{thm1}
		For any indices $1 \leq n < m \leq \infty$, there exists a natural equivalence of categories
		\[
			A_{m/n} \cln \Rep_{l_n}(\GL_N(F_n)) \isomto \Rep_{l_n}(\GL_N(F_m)).
		\]
		This $\{A_{m/n} \mid 0\leq n < m \leq \infty \}$ makes the diagram
		\begin{equation}\label{diag1}
			\begin{split}
				\xymatrix{
					\Rep_{l_n}(\GL_N(F_n)) \ar@{->}[r]^{A_{m'/n}} \ar@{->}[d]_{A_{m/n}} &
						\Rep_{l_n}(\GL_N(F_{m'})) \ar@{^(->}[dd]\\
					\Rep_{l_n}(\GL_N(F_m)) \ar@{^(->}[d] & \\
					\Rep_{l_m}(\GL_N(F_m)) \ar@{->}[r]_{A_{m'/m}} & \Rep_{l_m}(\GL_N(F_{m'}))
				}
			\end{split}
		\end{equation}
		commute for any $n \leq m \leq m'$.
		We also denote by $A_{m/n}$ the bijection
		\[
			\sA_{l_n}(\GL_N(F_n)) \isomto \sA_{l_n}(\GL_N(F_m))
		\]
		induced by the equivalence $A_{m/n}$.
		Then we can take the direct limit of $\{A_{\infty/n}\}_n$:
		\[
			\varinjlim_{n}A_{\infty/n} \cln \varinjlim_{n} \sA_{l_n}(\GL_N(F_n)) \isomto \sA(\GL_N(F_\infty)),
		\]
		which is also bijective.
	\end{thm}
	
	Next, we shall prove that $A_{m/n}$ coincides with the base change lifting.
	For a cyclic extension $F'/F$ of prime degree, let
	\[
		\BC_{F'/F} \cln \sA(\GL_N(F)) \to \sA(\GL_N(F'))
	\]
	be the base change lifting
	in the sense of \cite[Chapter 1, Section 6]{AC1989}.
	For a general cyclic extension $F'/F$ of finite degree, we define $\BC_{F'/F}$ as the composite of the base changes attached to intermediate exetnsions of $F'/F$ of prime degree.
	In particular, we write $\BC_{F_m/F_n} = \BC_{m/n}$.
	We denote by $\sA_l(\GL_N(F))$ the subset of $\sA(\GL_N(F))$ consisting of representations
	which have a non-trivial $\K_l(F)$-fixed vector.
	
	In the rest of this section, we suppose that $E/F$ is cyclic.
	We put $\Gamma = \Gal(E/F)$ and denote by $\Gamh$ the group of smooth characters of $\Gamma$ with valued in $\C^\times$.
	By local class field theory, we identify an element of $\Gamh$
	with a character $F^\times \to \C^\times$ which factors through $F^\times/N_{F_n/F}(F_n^\times)$ for some $n$.
	\begin{thm}\label{thm2}
		We take a sequence $\{l'_n\}_{n=1}^\infty$ satisfying the condition {\rm (L)} and
		such that there exists a positive integer $n_0$ such that
		$l'_n < 2^{-N}\lfloor p^{-1}(p-1)\psi_{\Fom/F}(b_n) \rfloor$ for any $n \geq n_0$.
		\begin{itemize}
		\item[(i)]
		For any indices $n_0 \leq n \leq m < \infty$,
		the bijection
		\[
			A_{m/n} \cln \sA_{l'_n}(\GL_N(F_n)) \isomto \sA_{l'_n}(\GL_N(F_m))
		\]
		coincides with the base change lifting $\BC_{m/n} = \BC_{F_m/F_n}$.
		\item[(ii)]
		For any $\pi \in \sA({\GL_N(F)})$, there exists an integer $n \geq 0$ such that
		\[\BC_{n/0}(\pi) \in \sA_{l'_n}(\GL_N(F_n)).\]
		\end{itemize}
	\end{thm}
	
	Now we can construct the base change lifting $\BC_{\infty/0}$ of infinite degree.
	\begin{defn}
		We define
		\[
		\BC_{\infty/0}\cln\sA({\GL_N(F)}) \to \sA({\GL_N(F_\infty)})
		\]
		by mapping $\pi$ to $A_{\infty/n}\circ \BC_{n/0}(\pi)$,
		where the $n$ is as in Theorem \ref{thm2} (ii).
	\end{defn}
	
	\begin{rmk}
		\begin{itemize}
		\item[(a)]
		By Theorem \ref{thm2} (i), the definition of $\BC_{\infty/0}$ is independent of the choice of $n$.
		\item[(b)]
		As noted above, at present, we can not avoid appealing to the local Langlands correspondence for $\GL_N$ over $F$ to prove Theorem \ref{thm2}.
		\item[(c)]
		The commutativity of the diagram in Theorem \ref{thm0} follows from
		\cite[Theorem 6.1]{ABPS2014} and the compatibility of $\BC$ with $\Res$ via $\LLC$.
		\end{itemize}
	\end{rmk}
	
	Furthermore, we study the structure of the fibers of $\BC_{\infty/0}$.
	Now we recall the Langlands sum following the exposition of \cite[Chapter 1]{HT2001}.
	We take a partition $(N_1,\ldots, N_r)$ of $N$.
	Let $\pi_i \in \sA(\GL_{N_i}(F))$ be an essentially square-integrable representation for each $1\leq i \leq r$.
	Let $s_i$ be the real number such that $|\cdot|^{s_i}$ is the absolute value of the central character of $\pi_i$.
	We reorder $\pi_1, \ldots, \pi_r$ so that $N_1^{-1}s_1 \geq \cdots \geq N_r^{-1}s_r$.
	We denote by $P(N_1,\ldots,N_r)$ the standard parabolic subgroup of $\GL_N(F)$
	whose Levi component is $\GL_{N_1}(F)\times\cdots\times\GL_{N_r}(F)$.
	Then the normalized induction
	\[
		\nInd_{P(N_1,\ldots,N_r)}^{\GL_N(F)}(\pi_1\boxtimes\cdots\boxtimes\pi_r)
	\]
	has a unique irreducible quotient, which we denote by
	$\pi_1 \Lsum \cdots \Lsum \pi_r$ and call the Langlands sum of $\pi_1,\ldots,\pi_r$.
	Each $\pi\in\sA(\GL_N(F))$ can be written as a Langlands sum and the $\pi_1,\ldots,\pi_r$ are uniquely determined up to a permutation.
	
	\begin{thm}\label{thm3}
		Let the notations and assumptions be as in Theorem \ref{thm2}.
		We suppose that $(p,N)=1$.
		\begin{itemize}
			\item[(i)]
			Let $\pi\in\sA(\GL_N(F))$ be an essentially square-integrable representation.
			We put $\pi_\infty = \BC_{\infty/0}(\pi)$.
			Let $\omega_\infty$ denote the central character of $\pi_\infty$.
			Then $\BC_{\infty/0}^{-1}(\pi_\infty)$ has a natural $\Gamh$-torsor structure
			and the map
			\[
			\omega\cln \BC_{\infty/0}^{-1}(\pi_\infty) \to \BC_{\infty/0}^{-1}(\omega_\infty)
			\]
			which maps $\pi'$ to its central character $\omega_{\pi'}$ is bijective.
			\item[(ii)]
			Let $\pi$ be any element of $\sA(\GL_N(F))$.
			We suppose that $p>N$.
			We can write 
			\begin{align*}
			 \pi = & \pi_1\Lsum (\pi_1\tensor \eta_{1,2})\cdots \Lsum (\pi_1\tensor \eta_{1,\mu_1}) & \\
				& \Lsum \cdots\\
				& \Lsum \pi_r\Lsum (\pi_r\tensor \eta_{r,2})\cdots \Lsum (\pi_r\tensor \eta_{r,\mu_r}),
			\end{align*}
			where 
			$\mu_i$ is an integer,
			$\pi_i \in \sA(\GL_{N_i}(F))$ is an essentially square-integrable representation
			for each $1\leq i \leq r$,
			and $\eta_{i,j}$ is an element of $\Gamh$ for each $1\leq i \leq r$ and $2\leq j \leq \mu_i$
			such that $\mu_1 N_1 + \cdots + \mu_r N_r = N$
			and the lifts $\BC_{\infty/0}(\pi_1),\ldots ,\BC_{\infty/0}(\pi_r)$ are all distinct.
			Then the group $\Gamh(\pi) = \Gamh^{\mu_1}\times\cdots\times\Gamh^{\mu_r}$ transitively acts on
			$\BC_{\infty/0}^{-1}(\pi_\infty)$.
			As a homogeneous space of $\Gamh(\pi)$, this is isomorphic to
			\[
				\Gamh(1,\eta_{1,2},\ldots,\eta_{1,\mu_1}) 
				\times\cdots\times
				\Gamh(1,\eta_{r,2},\ldots,\eta_{r,\mu_r}).
			\]
			Here, for $(\eta_1,\ldots,\eta_\mu) \in \Gamh^\mu$,
			we denote by $\Gamh(\eta_1,\ldots,\eta_r)$ the quotient of $\Gamh^\mu$
			by the following equivalence relation:
			Two elements $(\xi_1,\ldots,\xi_\mu)$ and $(\theta_1,\ldots,\theta_\mu)$ in $\Gamh^\mu$
			are equivalent if there exists a permutation $\sigma$ of $\{1,\ldots,\mu\}$ such that
			$\eta_j\xi_j = \eta_{\sigma(j)}\theta_{\sigma(j)}$ for each $j$.
	\end{itemize}
	\end{thm}
		
	\begin{rmk}
		We denote the local reciprocity map of $F$ by $\rec_F\cln W_F \to F^\times$.
		For $\phi\in\Phi(\GL_N(F))$, let $\chi_\phi$ denote the determinant character of $\phi$.
		If $p>N$, then Theorem \ref{thm3} shows that, using $\LLC_{F_\infty}$, we can characterize $\LLC_F$ as a map which makes the diagram
		\[
		\xymatrix{
			\Hom(F^\times,\C^\times) \ar[d]^{\rec_F^*}
			& \sA(\GL_N(F)) \ar[l]_(0.45){\omega} \ar[r]^{\BC_{\infty/0}} \ar[d]^{\LLC_F}
			& \sA(\GL_N(F_\infty)) \ar[d]^{\LLC_{F_\infty}} \\
			\Hom(W_F,\C^\times) & \Phi(\GL_N(F)) \ar[l]_(0.45){\chi} \ar[r]^{\Res_{\infty/0}} & \Phi(\GL_N(F_\infty))
		}
		\]
		commute and has the following properties:
		\begin{itemize}
			\item
			a Steinberg representation $\St_m(\sigma)$ maps to the outer tensor product
			\[
				\LLC_F(\sigma)\boxtimes \Sym^{m-1}\Std,
			\]
			where $\Std$ is the standard representation of $\SL_2(\C)$, and
			\item
			a Langlands sum maps to the corresponding direct sum.
		\end{itemize}
	\end{rmk}
	
	
	
	\section{Key lemmas}\label{KeyLemma}
	In this section, we prove an important lemma,
	which is a statement in Galois side corresponding to
	Theorem 1.3 (i) in the automorphic side.
	This is a compatibility of
	the restriction functor of Galois groups with respect to a finite totally ramified extension
	and Deligne's theory of close fields (\cite{Del1984}). 
	We recall Deligne's theory.
	Let $K$ be a local field with a finite residue field
	and $l$ a positive integer.
	We denote the ring of integers of $K$ by $\cO$ and the maximal ideal of $\cO$ by $\fkp$.
	We denote by $\Tr_l(K)$ the triple $(\cO/\fkp^l,\fkp/\fkp^{l+1},\veps)$ attached to $K$,
	where $\veps$ is the composite of the natural maps $\fkp/\fkp^{l+1} \to \fkp/\fkp^l \to \cO/\fkp^l$.
	We fix a separable closure $\Kb$ of $K$.
	Let $\Ext(K)^l$ denote the category of finite separable field extensions $K'$ of $K$ contained in $\Kb$
	such that $\Gal(\Kb/K') \supset \Gal(\Kb/K)^l$.
	Then we can construct a natural equivalence of categories
	\[
	T_K^l \cln \Ext(K)^l \isomto \Ext(\Tr_l(K))^l,
	\]
	where $\Ext(\Tr_l(K))^l$ is the category whose objects are extensions of $\Tr_l(K)$
	which satisfy the condition $\mathrm{C}^l$ in \cite[1.5.4]{Del1984}
	and morphisms are $R(l)$-equivalence classes (\cite[2.3]{Del1984}) of morphisms of $\Ext(\Tr_l(K))$.
	For an object $K'$ of $\Ext(K)^l$, $T_K^l(K')$ is defined to be the extension of triples $\Tr_l(K) \to \Tr_{lr}(K')$
	attached to the field extension $K'/K$, where $r$ is the ramification index of $K'/K$.
	
	We take another local field $K_1$ with finite residue field
	and denote the ring of integers of $K_1$ by $\cO_1$
	and the maximal ideal of $\cO_1$ by $\fkp_1$.
	Recall that $K$ and $K_1$ are called {\it $l$-close}
	if there exists an isomorphism of rings $\cO_1/\fkp_1^l \isomto \cO/\fkp^l$.
	Then we can construct an isomorphism of triples $\gamma \cln \Tr_l(K_1) \isomto \Tr_l(K)$.
	By mapping an extension $\Tr_l(K) \to X$ to 
	$
	\Tr_l(K_1) \xrightarrow{\gamma} \Tr_l(K) \to X
	$
	of $\Tr_l(K_1)$, we obtain an equivalence of categories
	\[
		\gamma^* \cln \Ext(\Tr_l(K))^l \to \Ext(\Tr_l(K_1))^l.
	\]
	
	Now let $L \subset \Kb$ be a finite totally ramified extension of $K$.
	We have
	\begin{equation}\label{Herbrand}
	\Gal(\Kb/L) \cap \Gal(\Kb/K)^u = W_L \cap \Gal(\Kb/K)^u = \Gal(\Kb/L)^{\psi_{L/K}(u)}
	\end{equation}
	\cite[1.1.2]{Win1983}.
	We denote by $i(L/K)$ the largest $i$ satisfying
	\[
		\Gal(\Kb/L)\Gal(\Kb/K)^i = \Gal(\Kb/K).
	\]
	Then for any integer $l \leq p^{-1}(p-1)i(L/K)$, the norm map
	$N_{L/K}$ induces an isomorphism of rings
	$\cO_L/\fkp_L^l \isomto \cO_K/\fkp_K^l$
	(see \cite[Proposition 2.2.1]{Win1983}).
	In particular, $K$ and $L$ are $l$-close.
	Hence there is a canonical isomorphism $\Tr_l(L) \isomto \Tr_l(K)$
	which sends the image of a uniformizer $\vpi_L$ of $L$ in $\fkp_L/\fkp_L^{l+1}$
	to that of $N_{L/K}(\vpi_L)$ in $\fkp_K/\fkp_K^{l+1}$.
	We denote the isomorphism of the triples by $\fkN_{L/K}$.
	
	Now we assume $l\leq p^{-1}(p-1)i(L/K)$.
	Then we have an equivalence of categories
	\[
		\fkN_{L/K}^* \cln \Ext(\Tr_l(K))^l \isomto \Ext(\Tr_l(L))^l.
	\]
	On the other hand, we have a functor
	$\Ext(K) \to \Ext(L)$ which maps an extension $K'$ of $K$ to the composite $K'L$.
	If $K'$ is an object of $\Ext(K)^l$, then by the equalities (\ref{Herbrand}), we have
		\begin{align*}
			\Gal(\Kb/K'L) &= \Gal(\Kb/K')\cap\Gal(\Kb/L) \\
			&\supset \Gal(\Kb/K)^l\cap\Gal(\Kb/L) \\
			&= \Gal(\Kb/L)^{\psi_{L/K}(l)} = \Gal(\Kb/L)^l.
		\end{align*}
	Thus $K'L$ is in $\Ext(L)^l$.
		
	Now we can prove the following lemma:
	
	\begin{lem}\label{ResDel}
		Suppose $l \leq (2p)^{-1}(p-1)i(L/K)$.
		Then the group isomorphism
		\[
			\fkN_{L/K*} \cln \Gal(\Kb/L)/\Gal(\Kb/L)^l \to \Gal(\Kb/K)\Gal(\Kb/K)^l
		\]
		induced by $\fkN_{L/K}$ coincides with the homomorphism which comes from
		the natural injection $\Gal(\Kb/L) \immto \Gal(\Kb/K)$.
	\end{lem}
	\begin{proof}
		We take a Galois object $K'$ of $\Ext(K)^l$.
		We put $L' = K'L$.
		We shall construct an isomorphism 
		\[
			\fkN'\cln T_L^l(L') \isomto \fkN_{L/K}^*T_K^l(K')	
		\]
		in $\Ext(\Tr_l(L))^l$ such that the following diagram is commutative:
		\begin{equation}\label{diag3}
		\begin{split}
			\xymatrix{
				\Gal(L'/L) \ar[r]^{\cdot|_{K'}} \ar[d]^{T_L^l} & \Gal(K'/K) \ar[d]^{T_K^l}\\
				\Aut_{\Tr_l(L)}(T_L^l(L')) \ar[rd]_{\ad(\fkN')} & \Aut_{\Tr_l(K)}(T_K^l(K')) \\
				 & \Aut_{\Tr_l(L)}(\fkN_{L/K}^*T_K^l(K')). \ar@{=}[u]\\
				}
		\end{split}
		\end{equation}
		Let $r$ denote the ramification index of $K'/K$.		
		We have $l \leq 2^{-1}i(L/K)$ and
		\[
			\Gal(\Kb/K') \supset \Gal(\Kb/K)^l \supset \Gal(\Kb/K)^{2^{-1}i(L/K)}.
		\]
		Hence we obtain inequalities
		\[
			\psi_{K'/K}^{-1}\left(\frac{1}{2}i(L/K)r\right) 
			\leq \frac{1}{2}i(L/K) + \frac{r-1}{r}\cdot\frac{1}{2}i(L/K) \leq i(L/K).
		\]
		Taking account of \[\Gal(\Kb/L)\Gal(\Kb/K)^{i(L/K)} = \Gal(\Kb/K),\] we have
		\begin{align*}
			& \Gal(\Kb/L')\Gal(\Kb/K')^{2^{-1}i(L/K)r} \\
			= &\Gal(\Kb/L')(\Gal(\Kb/K)^{\psi_{K'/K}^{-1}(2^{-1}i(L/K)r)}\cap\Gal(\Kb/K')) \\
			\supset& \Gal(\Kb/L')(\Gal(\Kb/K)^{i(L/K)}\cap\Gal(\Kb/K')) \\
			=& \Gal(\Kb/L')\cap\Gal(\Kb/K)^{i(L/K)} \\
			=& \Gal(\Kb/K')\cap(\Gal(\Kb/K)^{i(L/K)}\Gal(\Kb/L')) \\
			=& \Gal(\Kb/K').
		\end{align*}
		Hence we obtain $\Gal(\Kb/L')\Gal(\Kb/K')^{2^{-1}i(L/K)r} = \Gal(\Kb/K')$.
		Thus we have $2^{-1}i(L/K)r \leq i(L'/K')$ and the norm map $N_{L'/K'}$ provides an isomorphism
		$\fkN_{L'/K'}\cln \Tr_{lr}(L') \isomto \Tr_{lr}(K')$,
		which makes the diagram
		\[
		\xymatrix{
			\Tr_l(K) \ar[r] \ar[d]^{\fkN_{L/K}} & \Tr_{lr}(K') \ar[d]^{\fkN_{L'/K'}}\\
			\Tr_l(L) \ar[r] & \Tr_{lr}(L')
		}
		\]
		commute.
		Thus $\fkN_{L'/K'}$ is in fact an isomorphism $T_L(L') \isomto \fkN_{L/K}^*T_K(K')$
		in $\Ext(\Tr_l(L))^l$.
		We put $\fkN' = \fkN_{L'/K'}$.
		
		The commutativity of the diagram (\ref{diag3}) follows from
		the equality $N_{L'/K'}\circ \sigma = \sigma\circ N_{L'/K'}$ for any $\sigma \in \Gal(L'/L)$.
		Lemma \ref{ResDel} follows from the diagram (\ref{diag3}).
	\end{proof}
	
	For any real number $l\geq 0$, we define
	\[
	\Phi_l(\GL_N(K)) = \{\phi\in \Phi(\GL_N(K)) \mid \Gal(\Kb/K)^l \subset \Ker\phi\}.
	\]
	By Lemma \ref{ResDel}, we obtain the first key lemma:
	\begin{lem}\label{ResDel2}
		Let $K$ be a local field with a finite residue field
		and $L$ a finite totally ramified extension of $K$.
		Then, for any $l<(2p)^{-1}(p-1)i(L/K)$, the restriction of L-parameters
		\[
		\begin{array}{rcl}
		\Phi_l(\GL_n(K)) & \to & \Phi_l(\GL_n(L)) \\
		\phi & \mapsto & \phi|_{W_L\times\SL_2(\C)}
		\end{array}
		\]
		coincides with the map
		\[
		\begin{array}{lrcl}
		\fkN_{L/K}^*\cln & \Phi_l(\GL_n(K)) & \to & \Phi_l(\GL_n(L)) \\
		& \phi & \mapsto & \phi\circ(\fkN_{L/K*}\times\id_{\SL_2(\C)}).
		\end{array}
		\]
	\end{lem}
	
	\section{Proof of Theorem \ref{thm1}}\label{proof1}
	In this section, we prove Theorem \ref{thm1}.
	For this, we recall two ingredients.
	One is an equivalence of $\Rep_l(\GL_N(L))$ and the category of representations of some Hecke algebra,
	where $L$ is a local field with a finite residue field.
	The other is Kazhdan's theory of close local fields \cite{Kaz1986}.
	
	We denote by $\sH_l(\GL_N(L))$ the algebra of
	compactly supported $\K_l(L)$-bi-invariant functions on $\GL_N(L)$ with values in $\C$
	whose product is the convolution $*_l$ with respect to the Haar measure $\mu_{\GL_N(L),l}$ on $\GL_N(L)$
	normalized by
	\[
		\mu_{\GL_N(L),l}(\K_l(L)) =1.
	\]
	The characteristic function $e_{\K_l(L)}$ of $\K_l(L)$ is the unity of $\sH_l(\GL_N(L))$.
	The category of $\sH_l(\GL_N(L))$-modules is denoted by $\Mod(\sH_l(\GL_N(L)))$.
	\begin{lem}[{\rm \cite[Corollaire 3.9 (ii)]{Ber1984}}]\label{CdeBer}
		The functor $V\mapsto V^{\K_l(L)}$ gives an equivalence of categories
		\[
			\Rep_l(\GL_N(L)) \to \Mod(\sH_l(\GL_N(L))).
		\]
	\end{lem}
	
	By using this, we can prove the following:
	\begin{lem}\label{corCdeBer}
		For $l \leq m$, the functor
		\begin{eqnarray*}
			\Mod(\sH_l(\GL_N(L))) & \to & \Mod(\sH_m(\GL_N(L))) \\
			W & \mapsto & (\sH_m(\GL_N(L)) *_m e_{\K_l(L)}) \tensor_{\sH_l(\GL_N(L))}W
		\end{eqnarray*}
		makes the diagram
		\[
			\xymatrix{
					\Rep_l(\GL_N(L)) \ar@{^(->}[r] \ar[d] & \Rep_m(\GL_N(L)) \ar[d] \\
					\Mod(\sH_l(\GL_N(L))) \ar[r] & \Mod(\sH_m(\GL_N(L)))
				}
		\]
		commute, where the two vertical arrows are the equivalences in Lemma \ref{CdeBer}
		and the top horizontal arrow is a natural injection.
	\end{lem}
	\begin{proof}
		The following proof is similar to that of \cite[Corollaire 3.9 (ii)]{Ber1984}.
		Throughout this proof, we put $G = \GL_N(L)$, $K_l = \K_l(L)$, $\sH_l = \sH_l(G)$ and $e_l = e_{K_l}$.
		Note that the $\C$-vector space $\sH_m *_m e_l$ has
		an $\sH_m$-$\sH_l$-bimodule structure via
		$
			(h_m,h_m'*_m e_l,h_l) \mapsto h_m*_m h_m' *_m h_l
		$
		for any $h_m,h_m'\in \sH_m$ and $h_l \in \sH_l$.
		Let  $(\pi,V)$ be any object of $\Rep_l(G)$.
		The map
		\[
			(\sH_m *_m e_l) \tensor_{\sH_l} V^{K_l} \to V^{K_m}
		\]
		defined by
		\[
			(h *_m e_l) \tensor v \mapsto  \int_G (h*_m e_l)(g)\pi(g)vd\mu_{G,m}(g)
		\]
		is a well-defined left $\sH_m$-module homomorphism.
		It suffices to show that this is an isomorphism.
		This is surjective since $\pi$ is an object of $\Rep_l(G)$.
		We denote by $\cN$ the kernel of the above homomorphism.
		Now let $\Mod_l(\sH_m)$ denote the full subcategory of $\Mod(\sH_m)$ consisting of objects $W$ which are generated by $e_l*W$. 
		Then the equivalence of categories of Lemma \ref{CdeBer} induces that of $\Rep_l(G)$ and $\Mod_l(\sH_m)$.
		This equivalence and Lemma \ref{CdeBer} imply that the latter is equivalent to $\Mod(\sH_l)$ and stable by sub-quotient.
		Since $\sH_m *_m e_l$ and $V^{K_m}$ are object of $\Mod_l(\sH_m)$, so is $\cN$.
		In addition, there is no non-trivial vectors on $\cN$ which is fixed by the left action of $e_l$.
		Therefore $\cN = 0$ and the above homomorphism is an isomorphism.
	\end{proof}

	Next we recall Kazhdan's theory.
	Let $F_1$ and $F_2$ be local fields of residual characteristic $p$ which are $l$-close.
	Let $\cO_i$ and $\fkp_i$ denote the ring of integers and the maximal ideal of $F_i$ $(i=1,2)$.
	Let $\alpha \cln \cO_2/\fkp_2^l \isomto \cO_1/\fkp_1^l$ be an isomorphism of rings.
	We fix a uniformizer $\vpi_2$ of $F_2$ and choose a lift $\vpi_1 \in \fkp_1$ of $\alpha(\vpi_2 \mod{\fkp_2})$.
	By the Cartan decomposition, the datum $(\alpha,\vpi_2,\vpi_1)$ gives a $\C$-linear isomorphism
	\[
		(\alpha,\vpi_2,\vpi_1)^* \cln \sH_l(\GL_N(F_1)) \isomto \sH_l(\GL_N(F_2))
	\]
	(see \cite{Kaz1986}).
	In Kazhdan's original paper, he showed that if $F_1$ and $F_2$ are sufficiently close then
	$(\alpha,\vpi_2,\vpi_1)^*$ is compatible with the convolution products.
	Lemaire showed a more precise result for $\GL_N$: 
	
	\begin{lem}[\rm{\cite[Proposition 3.1.1]{Lem2001}}]\label{Lem311}
		If $F_1$ and $F_2$ are $l$-close,
		the isomorphism $(\alpha,\vpi_2,\vpi_1)^*$ is compatible with the convolution products.
		Hence it is a $\C$-algebra isomorphism.
	\end{lem}
	
	Now we prove Theorem \ref{thm1}.
	Let $\Fom/F$ be an infinite APF extension.
	For any indices $1\leq n < m\leq \infty$, 
	we have 
	\[
		l_n \leq \frac{p-1}{p}\psi_{E/F}(b_n) = \frac{p-1}{p}i(F_{m}/F_n).
	\]
	Here, we use equalities $\psi_{E/F}(b_n) = i(F_{n+1}/F_n) = i(E/F_n)$ 
	(see \cite[1.4.1 (b)]{Win1983})
	and inequalities $i(E/F_n) \leq i(F_m/F_n) \leq i(F_{n+1}/F_n)$
	(see \cite[Proposition 1.2.3]{Win1983}).
	Thus the norm map with respect to $F_m/F_n$ induces an isomorphism of rings 
	\[
		\alpha_{m/n}\cln \cO_{F_m}/\fkp_{F_m}^{l_n} \isomto \cO_{F_n}/\fkp_{F_n}^{l_n}
	\]
	\cite[Proposition 2.2.1]{Win1983}.
	We fix a uniformizer $\vpi_m$ of $F_m$.
	By Lemma \ref{Lem311}, we obtain an isomorphism of $\C$-algebras
	\[
		\beta_{m/n}^* = (\alpha_{m/n},\vpi_m,N_{F_m/F_n}(\vpi_m))^* \cln \sH_{l_n}(\GL_N(F_n)) \isomto \sH_{l_n}(\GL_N(F_m)).
	\]	
	 By Lemma \ref{CdeBer}, this induces an equivalence of categories
	\[
		A_{m/n} \cln \Rep_{l_n}(\GL_N(F_n)) \isomto \Rep_{l_n}(\GL_N(F_m)).
	\]
	The transitivity of norm maps implies that the commutativity of the following diagram
	\[
		\xymatrix{
				\sH_{l_n}(\GL_N(F_n)) \ar[r]^{\beta_{m/n}^*} \ar[rd]_{\beta_{m'/n}^*} &
					\sH_{l_n}(\GL_N(F_m))  \ar[d]^{\beta_{m'/m}^*} \\
				 & \sH_{l_n}(\GL_N(F_{m'})).
			}
	\]	
	This and Corollary \ref{corCdeBer} show the commutativity of the diagram (\ref{diag1}).
	This completes the proof of Theorem \ref{thm1}.
	
	\section{Proof of Theorem \ref{thm2}}\label{proof2}
	In this section, we prove Theorem \ref{thm2}.
	
	Let $F$ be a mixed characteristic local field.
	We denote the local Langlands correspondence for $\GL_N$ over $F$ by
	\[
	\LLC_F \cln \sA(\GL_N(F)) \to \Phi(\GL_N(F)).
	\]
	By \cite[Proposition 4.2]{ABPS2014}, we have
	\[
		\LLC_F(\sA_l(\GL_N(F))) \subset \Phi_l(\GL_N(F)).
	\]	
	
	Let $F_1$ and $F_2$ be local fields with finite residue fields
	which are $l$-close.
	We choose a datum $\beta = (\alpha,\vpi_2,\vpi_1)$ as in Kazhdan's theory.
	Then we obtain Kazhdan's correspondence $\beta^*\cln \sA_l(\GL_N(F_1)) \isomto \sA_l(\GL_N(F_2))$.
	Moreover, from $\beta$ we can canonically define an isomorphism of triples $\gamma\cln \Tr_l(F_2) \to \Tr_l(F_1)$.
	The following compatibility of $\beta^*$ with $\gamma^*$ via the local Langlands correspondence was proved
	by Aubert, Baum, Plymen and Solleveld in their preprint \cite{ABPS2014}.
	
	\begin{thm}[{\cite[Theorem 6.1]{ABPS2014}}]\label{ABPSThm61}
		Let $l'$ be any integer such that $0 \leq l' < 2^{-N}l$.
		Then the following diagram is commutative:
		\[
			\xymatrix{
					\sA_{l'}(\GL_N(F_1)) \ar[r]^{\beta^*} \ar[d]^{\LLC} & \sA_{l'}(\GL_N(F_2)) \ar[d]^{\LLC} \\
					\Phi_{l'}(\GL_N(F_1)) \ar[r]^{\gamma^*} & \Phi_{l'}(\GL_N(F_2)).
				}
		\]
	\end{thm}
	
	Now let us prove Theorem \ref{thm2} (i).
	By Theorem \ref{ABPSThm61}, the map $\beta_{m/n}^*$ in Section \ref{proof1} is compatible
	with the map $\fkN_{F_m/F_n}^*$ in Lemma \ref{ResDel2} via $\LLC$.
	Now we have inequalities
	$\l_n' \leq 2^{-N} p^{-1} (p-1) i(F_m/F_n) \leq (2p)^{-1}(p-1) i(F_m/F_n)$.
	Hence, by Lemma \ref{ResDel2}, the map $\fkN_{F_m/F_n}^*$ coincides with
	the map induced by the restriction $W_{F_m} \immto W_{F_n}$.
	Since the latter map is compatible with $\BC_{m/n}$ via $\LLC$,
	we have completed the proof.
	
	Next we show Theorem \ref{thm2} (ii).
	By the local Langlands correspondence and Theorem \ref{ABPSThm61},
	this is also reduced to showing the corresponding assertion on Galois representations.
	Thus we shall show that for any $\phi\in\Phi(\GL_N(F))$ there exists $n$ such that
	$\phi|_{W_{F_n}} \in \Phi_{l_n'}(\GL_N(F_n))$.
	Take any $\phi\in \Phi(\GL_N(F))$.
	Then we have $\phi\in \Phi_l(\GL_N(F))$ for some $l$.
	By the equality (\ref{Herbrand}) in Section \ref{KeyLemma}, we have
	$W_{F_n} \cap \Gal(\Fb/F)^l = \Gal(\Fb/F_n)^{\psi_{\Fom/F}(l)}$ for any $n$.
	Since $l_n' \to \infty$ as $n \to \infty$,
	there exists an integer $n$ such that $\psi_{\Fom/F}(l) \leq l_n'$.
	Thus $\phi|_{W_{F_n}}$ is trivial on $\Gal(\Fb/F_n)^{l_n'}$
	i.e. $\phi|_{W_{F_n}} \in \Phi_{l_n'}(\GL_N(F_n))$, as claimed.
	
	\section{Proof of Theorem \ref{thm3}}\label{proof3}
	Finally, we prove Theorem \ref{thm3}.
	First, we show (i) for a supercuspidal $\pi$.
	We put $\pi_\infty = \BC_{\infty/0}(\pi)$.
	The fiber $\BC_{\infty/0}^{-1}(\pi_\infty)$ has a $\Gamh$-set structure via
	$\pi' \mapsto \pi'\tensor \eta$, where $\pi'\in \BC_{\infty/0}^{-1}(\pi_\infty)$ and $\eta \in \Gamh$.
	We shall show that this action is simply transitive.
	The assumption $(p,N)=1$ shows that the $\Gamh$-action is simple.
	Let us prove the transitivity.
	We take any $\pi' \in \BC_{\infty/0}^{-1}(\pi_\infty)$.
	By Theorem \ref{thm2} (ii), we can take an integer $n$ such that
	both $\BC_{n/0}(\pi)$ and $\BC_{n/0}(\pi')$ belong to $\sA_{l_n}(\GL_N(F_n))$.
	Since $\BC_{\infty/0} = A_{\infty/n}\circ\BC_{n/0}$ and $A_{\infty/n}$ is injective,
	we have $\BC_{n/0}(\pi) = \BC_{n/0}(\pi')$.
	It suffices to show that there exists a smooth character $\eta\cln F^\times \to \C^\times$
	which factors through $F^\times/N_{F_n/F}(F_n^\times)$ such that $\pi' \simeq \pi\tensor\eta$.
	We show this by induction on $n$.
	The case $n=1$ is \cite[Chapter 1, Proposition 6.7]{AC1989}.
	We assume that the assertion holds for $n-1$.
	By the case $n=1$, we can find a smooth character $\eta_1\cln F_{n-1}^\times \to \C^\times$
	which factors through $F_{n-1}^\times/N_{F_n/F_{n-1}}(F_n^\times)$ and satisfies
	$\BC_{n-1/0}(\pi') \simeq \BC_{n-1/0}(\pi)\tensor \eta_1$.
	Let $\omega_\pi$ (resp. $\omega_{\pi'}$) denote the central character of $\pi$ (resp. $\pi'$).
	Then we have
	\[
	 \omega_{\pi'}\circ N_{F_{n-1}/F} = (\omega_\pi\circ N_{F_{n-1}/F}) \eta_1^N.
	\]
	Thus we obtain
	\[
		\eta_1^N = (\omega_{\pi'}\omega_\pi^{-1})\circ N_{F_{n-1}/F}.
	\]
	By the assumption that $(p,N)=1$, we find a character $\eta_1'$ on $F^\times$ such that
	\[
		\eta_1 = \eta_1'\circ N_{F_{n-1}/F}.
	\]
	Hence we have $\BC_{n-1/0}(\pi') \simeq \BC_{n-1/0}(\pi\tensor \eta_1')$ and by the induction hypothesis
	there exists a smooth character $\eta_{n-1}$ on $F^\times$ which is trivial on $N_{F_{n-1}/F}(F_{n-1}^\times)$ and satisfies
	$\pi' \simeq \pi\tensor(\eta_1'\eta_{n-1})$.
	Then $\eta = \eta_1'\eta_{n-1}$ is the requested character.
	
	Taking central character maps $\pi\tensor \eta$ to the character $\omega_\pi \eta^N$.
	By the assumption $(p,N)=1$, this gives a bijection
	$\BC_{\infty/0}^{-1}(\pi_\infty) \to \BC_{\infty/0}^{-1}(\omega_\infty)$.
	
	Now we show (i) for any essentially square-integrable $\pi$.
	Then there exist a unique divisor $m$ of $N$ and a unique supercuspidal representation $\sigma \in \sA(\GL_{N/m}(F))$ such that $\pi$ is equivalent to the unique irreducible quotient $\St_m(\sigma)$ of
	\[
		\nInd_{P(N/m,\ldots,N/m)}^{\GL_N(F)}(\sigma\otimes|\det|^{(1-m)/2}\boxtimes\cdots\boxtimes
			\sigma\otimes|\det|^{(m-1)/2})
	\]
	(\cite[Theorem 9.3]{Zel1980}).
	We put $\sigma_\infty = \BC_{\infty/0}(\sigma)$.
	Let us show that the map
	\begin{eqnarray*}
		\BC_{\infty/0}^{-1}(\sigma_\infty) & \to & \BC_{\infty/0}^{-1}(\pi_\infty) \\
		\sigma' & \mapsto & \St_m(\sigma')
	\end{eqnarray*}
	is bijective.
	Its well-definedness follows from \cite[Lemma 6.12]{AC1989},
	\cite[Th\'eor\`eme 2.17 (c)]{Bad2002} and \cite[Proposition A.4.1]{HL2006}.
	Its injectivity follows from the uniqueness of the expression $\St_m(\sigma')$.
	We show its surjectivity.
	Take any $\pi' \in \BC_{\infty/0}^{-1}(\pi_\infty)$.
	Then by \cite[Th\'eor\`eme 2.17 (c)]{Bad2002} and \cite[Proposition A.4.1]{HL2006},
	we have $\BC_{n/0}(\pi') = \St_m(\BC_{n/0}(\sigma))$ for some $n$.
	Since $\pi'$ is essentially square-integrable, there exists a divisor $m'$ of $N$ and a supercuspidal representation $\sigma'\in \sA(\GL_{N/m'}(F))$ such that $\pi' = \St_{m'}(\sigma')$.
	The assumption $(p,N)=1$ and \cite[Lemma 6.12]{AC1989} show that
	$\BC_{n/0}(\pi') = \St_{m'}(\BC_{n/0}(\sigma'))$.
	Hence we have $m'=m$ and
	$\sigma' \in \BC_{n/0}^{-1}(\BC_{n/0}(\sigma)) \subset \BC_{\infty/0}^{-1}(\sigma_\infty)$.
	Therefore the surjectivity follows, as claimed.
	
	The statement (ii) follows from the uniqueness of the Langlands sum
	and the fact that the functor $A_{\infty/n}$ preserves the Langlands sum (\cite[Proposition A.4.1]{HL2006}).
	
	
	\section*{Acknowledgments}
	The author would like to thank his supervisor Professor Yuichiro Taguchi
	for suggesting the problem and giving many stimulating comments.
	He is also grateful to Professor Yoichi Mieda
	for useful discussion, in particular pointing out the importance of Theorem \ref{thm2} (i).
	This study was supported by JSPS KAKENHI, Grant-in-Aid for JSPS fellows, Grant number 13J04056.

\end{document}